\documentclass[12pt, twoside, leqno]{article}

\usepackage[margin=1.4in]{geometry}
\usepackage{amsthm}
\usepackage{amssymb,latexsym}
\usepackage{enumerate}
\usepackage{amsfonts,amsmath,mathrsfs}
\usepackage{topcapt}
\usepackage[cp1250]{inputenc}
\usepackage[T1]{fontenc}
\usepackage{tabularx}
\usepackage{multirow}
\usepackage{epsfig}
\usepackage[figuresright]{rotating}
\usepackage{hyperref}
\hypersetup{hypertexnames=false}
\let\originalforall=\forall
\renewcommand{\forall}{\mathop{\vcenter{\hbox{\Large$\originalforall$}}}}

\let\originalexists=\exists
\renewcommand{\exists}{\mathop{\vcenter{\hbox{\Large$\originalexists$}}}}
\newtheorem{de}{Definition}

\newtheorem{tw}[de]{Theorem}
\newtheorem{prop}[de]{Proposition}
\newtheorem{rem}[de]{Remark}
\newtheorem{lem}[de]{Lemma}

\title{On the relationships between Fourier - Stieltjes coefficients and spectra of measures}
\begin{document}
\baselineskip=17pt
\author{Przemysław Ohrysko \\
Institute of Mathematics, Polish Academy of Sciences\\
00-956 Warszawa, Poland\\
E-mail: p.ohrysko@gmail.com
\and
Michał Wojciechowski\\
Institute of Mathematics, Polish Academy of Sciences\\
00-956 Warszawa, Poland\\
E-mail: miwoj-impan@o2.pl}

\date{}

\maketitle

\renewcommand{\thefootnote}{}



\renewcommand{\thefootnote}{\arabic{footnote}}
\setcounter{footnote}{0}
\begin{abstract}
We construct examples of uncountable compact subsets of complex numbers with the property that any Borel measure on the circle group taking values of its Fourier coefficients
from this set has a natural spectrum. For measures with Fourier coefficients tending to 0 we construct an open set with this property. We also give an example of a singular measure whose spectrum is contained in our set.
\end{abstract}

\section{Introduction}
Let $M(\mathbb T)$ denote the convolution algebra of Borel
measures on the unit circle group. For details of notation and
basic definitions see [Ka]. Closure of the values of the Fourier
coefficients of $\mu\in M(\mathbb T)$ is obviously a subset of the
spectrum $\sigma(\mu)$ of $\mu$. However, as was observed by
Wiener and Pitt (see [WP]) in general it is a proper subset. There
are several different proofs of this phenomenon - cf. [S], [G].
The \textit{Wiener - Pitt phenomenon} is equivalent to the
\textit{inversion problem} which states that assumption
$|\widehat{\mu}(n)|>c>0$ for all $n\in\mathbb{Z}$ and constant $c$
does not ensure the invertibility of $\mu$ as an element in Banach
algebra $M(\mathbb{T})$. Moreover, it is closely related to the
asymmetry of the algebra $M(\mathbb{T})$ which is presented in
[R].

On the other hand there are classes of Borel measures for which
the spectrum equals the closure of the set of Fourier
coefficients. Such measures are said to have a natural spectrum.
It is known that absolutely continuous and purely discrete
measures have a natural spectrum (cf. [Z1]). A natural question is
how to recognize a measure with a natural spectrum using only the
information about its Fourier coefficients. Motivated by this
problem we introduce the notion of Wiener - Pitt sets.  We say
that a compact set $A\subset \mathbb C$ is a {\it Wiener - Pitt
set} whenever $\widehat{\mu}(\mathbb Z)\subset A$ implies that
$\mu$ has a natural spectrum.

Finite sets are easy examples of Wiener - Pitt sets. Indeed, if
$A=\{a_1,\dots, a_k\}$ then by the Gelfand theory, the polynomial
$P(z)=(z-a_1)\dots(z-a_k)$ satisfies $\widehat{P(\mu)}(n)=0$ for
every $n\in\mathbb Z$. Therefore $P(\mu)=0$, which in turn yields
that $P(\psi(\mu))=0$ for every linear multiplicative functional
$\psi$. Hence $\psi(\mu)$ is a root of polynomial $P$ and
therefore $\sigma(\mu)\subset A$. Finite sets are the only known
class of Wiener - Pitt sets. The aim of this paper is to construct
the infinite (even uncountable) Wiener - Pitt compacta. Our
construction gives a quite flexible family of zero dimensional
examples and, moreover, the examples  are stable under suitable
small perturbations. Furthermore for continuous measures we
construct the open subset $U\subset\mathbb{C}$ such that
$0\in\overline{U}$ and any continuous measure with
$\widehat{\mu}(\mathbb{Z})\subset U\cup \{0\}$ satisfies $\mu\in
\mathrm{Rad}(L^{1}(\mathbb{T}))$ which thanks to the theorem of
Zafran (see Theorem $\ref{zaf}$ in section 4) gives Theorem 1
formulated below.

\begin{tw}
There exists an open set $U\subset\mathbb{C}$ with $0\in\overline{U}$ such that
every continuous measure $\mu$ with $\widehat{\mu}(\mathbb{Z})\subset U\cup \{0\}$
has a natural spectrum.
\end{tw}
In general case we prove the following.
\begin{tw}
There exists a set $K$ homeomorphic to the Cantor set such that $0\in K$ and
every measure $\mu$ with $\widehat{\mu}(\mathbb{Z})\subset K$ has a natural spectrum.
\end{tw}

To complete the above results we provide the example of a singular
measure with the spectrum contained in the set $U$ constructed in
Theorem 1. The example, interesting in itself, is given in Section
6 and combines the techniques of the Riesz products with Rudin -
Shapiro polynomials to get a singular measure with coefficients
smaller then from any sequence tending to 0 sufficiently fast
taken in advance.

The construction is quite involved - it uses four main ingredients. First is the Zafran characterization of measures with Fourier coefficients tending to zero with a natural spectrum.
Second is the Katznelson - DeLeeuw theorem which is the main ingredient of their qualitative version of Grużewska - Rajchman theorem (nevertheless, in the paper we use a stronger and more involved result from [GM] - just to avoid unnecessary complications). Third is the Bożejko - Pełczyński
theorem on the uniform invariant approximation property of $L^1(\mathbb T)$. The Fourth ingredient is the Littlewood Conjecture proved by McGehee Pigno Smith and independently by Konyagin.

The main construction constituting the proof of Theorem 1 under the additional assumption that the regarding measure has Fourier coefficients tending to 0, is presented in Section 3. The aim of Section 4 is to prove that this additional assumption can be omitted. In Section 5 we complete the proof of Theorem 1 and we show how Theorem 2 could be derived from Theorem 1. Section 2 is devoted
to prove of auxiliary lemmas. Here we combine two main analytical ingredients, the Bożejko - Pełczyński uniform invariant approximation property and the Littlewood conjecture,  to derive Lemma 8 - the main tool used in further inductions.
\begin{rem}
Some results which are presented in this paper are contained in first author's MA thesis (see $[O]$).
\end{rem}

\section{Preparatory lemmas}

In this section we prove crucial lemmas which are also of
independent interest. The following definition will be useful.
\begin{de}
For $\mu\in M(\mathbb{T})$ and $\varepsilon>0$ we write $\mu\in
F(\varepsilon)$ iff $|\widehat{\mu}(n)|<\varepsilon$ for all
$n\in\mathbb{Z}$. Set of trigonometric polynomials will be denoted
$\mathscr{P}$. Abbreviation $\#f$ will be used for number of
elements in support of $\widehat{f}$ (for $f\in\mathscr{P}$), i.
e.
$\#f=\#\mathrm{supp}\widehat{f}=\#\{n\in\mathbb{Z}:\widehat{f}(n)\neq
0\}$. Also, for $f\in \mathscr{P}$ we say that $f\in G(a)$ for
some positive number $a$ when inequality $|\widehat{f}(n)|\geq a$
is true for all integers $n$ such that $\widehat{f}(n)\neq 0$.
\end{de}

We will use in the sequel two powerful results. First is the
Littlewood conjecture (for a proof consult [MPS] and [Ko])

\begin{tw}\label{MPS+K}{\rm (McGehee, Pigno, Smith; Konyagin)}
For every $f\in\mathscr{P}$ of the form
\begin{equation*}
f(t)=\sum_{k=1}^{N}c_{k}e^{in_{k}t},
\end{equation*}
where $n_{k}$ is the sequence of increasing integers and $|c_{k}|\geq
1$, $1\leq k\leq N$, we have
\begin{equation*}
||f||_{L^{1}(\mathbb{T})}>L\ln N,
\end{equation*}
where the constant $L>0$ does not depend on $N$.
\end{tw}

The Second fact is the invariant uniform approximation property of
$L^1(\mathbb{T})$ (proofs are contained in papers [BP] and [B] or
in book [Wo]).

\begin{tw}[Bożejko, Pełczyński; Bourgain]\label{bope}
Let $\Lambda\subset\mathbb{N}$ be a finite set with $\#\Lambda=k$.
Then for every $\varepsilon>0$ there exists $f\in\mathscr{P}$ such
that
\begin{enumerate}
    \item $\widehat{f}(n)=1$ for $n\in\Lambda$.
    \item $||f||_{L^{1}(\mathbb{T})}\leq 1+\varepsilon$
    \item $\#\{n\in\mathbb{N}:\widehat{f}(n)\neq
    0\}\leq\left(\frac{\alpha}{\varepsilon}\right)^{2k}$ for some
    $\alpha>0$.
\end{enumerate}
We will write $f\in BPB_{\varepsilon}(\Lambda)$ for polynomials
with the described properties.
\end{tw}

It is now time to formulate our first lemma.

\begin{lem}\label{postep}
Let $f\in \mathscr{P}$. If $\#f\geq d$ for some positive number $d$, then
there exists a two-sided arithmetical progression
$\Gamma\subset\mathbb{Z}$ such that

\begin{equation*}
d\leq\#\left(\mathrm{supp}\widehat{f}\cap{\mathbf
1}_\Gamma\right)<2d.
\end{equation*}
\end{lem}

\begin{proof}
If $d\leq\#f<2d$, we take $\Gamma=\mathbb{Z}$. Otherwise $\#f\geq 2d$
and taking $\Gamma_{1}=2\mathbb{Z}$ and
$\Gamma_{2}=2\mathbb{Z}+1$ we get
$d\leq\#\left(\mathrm{supp}{\widehat{\mu}}\cap{\mathbf 1}_{\Gamma_{i}}
\right)$ for some $i=1,2$. If moreover
$\#\left(\mathrm{supp}\widehat{f}\cap{\mathbf
1}_{\Gamma_{i}}\right)<2d$, we put $\Gamma=\Gamma_i$. Otherwise we
repeat this procedure
\end{proof}

The second lemma is much more sophisticated (since for absolutely
continuous measures the norm in $M(\mathbb{T})$ is equal to the
norm in $L^{1}(\mathbb{T})$ of the density with respect to the
Lebesgue measure we will use this two notations interchangeably).

\begin{lem}\label{gl1}
There exists a function $\varepsilon=\varepsilon(K,a)$ and $c>0$
such that whenever $||f+\nu||_{M(\mathbb{T})}<K$ for some $f\in
G(a)$ and $\nu\in F(\varepsilon)$, then $\#
f<\exp\left(c\frac{K}{a}\right)$.
\end{lem}

\begin{proof}
Let us put $d=\exp(\frac{4K}{aL})$ and define the function $\varepsilon=\varepsilon(K,a)$ by the formula
\begin{equation}\label{eps}
\varepsilon(K,a)=\frac{aL\ln
d}{4\alpha^{2d}}=K\exp\left(-2\ln(\alpha)\exp\left(\frac{4K}{aL}\right)\right)
\end{equation}
where $\alpha$ is the constant from Theorem $\ref{bope}$. We will
show that the assumption $\#f\geq d$ leads to a contradiction
which proves our lemma with constant $c=\frac{4}{L}$. By Lemma
\ref{postep} there exists $\Gamma\subset\mathbb{Z}$ such that
$d\leq\#\left(\mathrm{supp}\widehat{\nu}\cap{\mathbf
1}_\Gamma\right)<2d$. We define $f_{1}\in \mathscr {P}$ and
$\nu_{1}\in M(\mathbb{T})$ by taking
$\widehat{f_{1}}=\widehat{f}\cdot{\mathbf 1}_{\Gamma}$ and
$\widehat{\nu_{1}}=\widehat{\nu}\cdot{\mathbf 1}_\Gamma$. Since
multiplying Fourier sequences by characteristic function of
$\Gamma$ corresponds to a convolution with the measure of norm
one, $||f_{1}+\nu_{1}||_{M(\mathbb{T})}<K$ holds. By the
definition of $\Gamma$, $\#f_{1}<2d$. It follows from Theorem
\ref{MPS+K} that $||f_{1}||_{M(\mathbb{T})}\geq a\cdot L\ln d$.
Let $\Theta\in BPN_1(\mathrm{supp}f_{1})$. Then
$||\Theta||_{L^{1}(\mathbb{T})}<2$, $\widehat{\Theta}(n)=1$ for
$n\in\mathrm{supp}f_{1}$, $\#\Theta<\alpha^{4d}$ for some
$\alpha>0$. By the triangle inequality,
\begin{equation*}
2K>||(f_{1}+\nu_{1})\ast\Theta||_{M(\mathbb{T})}=||f_{1}+\nu_{1}\ast\Theta||_{M(\mathbb{T})}\geq
||f||_{L^{1}(\mathbb{T})}-||\Theta\ast\nu_{1}||_{M(\mathbb{T})}.
\end{equation*}
Estimating the $L^1$-norm by the $L^2$-norm we get
$||\Theta\ast\nu_{1}||_{L^{1}(\mathbb{T})}\leq
||\Theta\ast\nu_{1}||_{L^{2}(\mathbb{T})}\leq 2\varepsilon \alpha^{2d}$.
This all together gives
$2K>a\cdot L\ln d-2\varepsilon \alpha^{2d}$.
Hence the formula for $\varepsilon=\varepsilon(K,a)$ from equation ($\ref{eps}$) leads to the announced contradiction.
\end{proof}

The next lemma gives more information.

\begin{lem}\label{gl2}
For every $\varepsilon >0$ there exists
$\delta=\delta(\varepsilon,a,K)$ such that if
$||f+\nu||_{M(\mathbb{T})}<K$ for $f\in G(a)$ and $\nu\in
F(\delta)$, then $||f||_{L^{1}(\mathbb{T})}<K(1+\varepsilon)$.
\end{lem}
\begin{proof}
Let $\Theta\in BPB_{\frac{\varepsilon}{2}}(\mathrm{supp}\widehat{f})$. Then
$||\Theta||_{L^{1}(\mathbb{T})}<1+\frac{\varepsilon}{2}$, $\widehat{\Theta}(n)=1$ for $n\in\mathrm{supp}\widehat{f}$,
$\#\Theta<\left(\frac{\lambda}{\varepsilon}\right)^{2\# f}=\exp(2\# f\ln\frac{\lambda}{\varepsilon})$ for some $\lambda >0$.
By Lemma \ref{gl1} for sufficiently small $\delta$ we have
\begin{equation*}
\#\Theta<\exp\left(2\ln\left(\frac{\lambda}{\varepsilon}\right)\exp(cKa^{-1})\right),
\end{equation*}
where $c$ is as in Lemma \ref{gl1}. By the triangle inequality
\begin{equation*}
\left(1+\frac{\varepsilon}{2}\right)K\geq||\Theta\ast(f+\nu)||_{L^{1}(\mathbb{T})}=||f+\Theta\ast\nu||_{L^{1}(\mathbb{T})}\geq
||f||_{L^{1}(\mathbb{T})}-||\Theta\ast\nu||_{L^{1}(\mathbb{T})}.
\end{equation*}
Since $|\widehat{\Theta}|<1+\frac{\varepsilon}{2}$, we have $\Theta\ast\nu\in
F((1+\frac{\varepsilon}{2})\delta)$. Obviously
\begin{equation*}
\#(\Theta\ast\nu)\leq\#\Theta<\exp\left(2\ln\left(\frac{\lambda}{\varepsilon}\right)(\exp(cK|a|^{-1}))\right).
\end{equation*}
Estimating the $L^1$- norm by the $L^2$-norm we get
\begin{equation*}
||\Theta\ast\nu||_{L^{1}(\mathbb{T})}\leq
||\Theta\ast\nu||_{L^{2}(\mathbb{T})}\leq
\left(1+\frac{\varepsilon}{2}\right)\delta\exp\left(\ln\left(\frac{\lambda}{\varepsilon}\right)\exp(cK|a|^{-1})\right).
\end{equation*}
Completing all together we get
\begin{equation*}
\left(1+\frac{\varepsilon}{2}\right)K>||f||_{L^{1}(\mathbb{T})}-\left(1+\frac{\varepsilon}{2}\right)\delta\exp\left(\ln\left(\frac{\lambda}{\varepsilon}\right)\exp(cK|a|^{-1})\right).
\end{equation*}
If we put
\begin{equation*}
\delta<\frac{\varepsilon}{1+\varepsilon}K\exp\left(-\ln\left(\frac{\lambda}{\varepsilon}\right)\exp(cKa^{-1})\right),
\end{equation*}
then the assumption
$||f||_{L^{1}(\mathbb{T})}>(1+\varepsilon)K$
leads to a contradiction.
\end{proof}

\section{The case of $M_{0}(\mathbb{T})$}

We will make use of the following proposition contained in [W]
(see Proposition 1.9 in this paper).
\begin{prop}\label{woj}
Suppose $A$ is a commutative Banach algebra with unit and $x\in A$ has a
finite spectrum,
$\sigma(x)=\{\lambda_{1},\lambda_{2},\ldots,\lambda_{n}\}$. Put
$\delta=\min_{i\neq j}|\lambda_{i}-\lambda_{j}|$. Then there
exist orthogonal idempotents $x_{1},x_{2},\ldots,x_{n}\in A$
(i.e. $x_{i}^{2}=x_{i}$ and $x_{i}x_{j}=0$ for $i\neq j$,
$i,j=1,\ldots,n$) such that
\begin{equation*}
x=\lambda_{1}x_{1}+\lambda_{2}x_{2}+\ldots+\lambda_{n}x_{n}.
\end{equation*}
Moreover, the inequality
$||x_{i}||\leq \delta^{-n+1}2^{n-1}||x||^{n-1}$
holds for $i=1,2,\ldots,n$.
\end{prop}
One abbreviation is useful when it comes to manipulate with
convolution powers. We will write $f^{m}=f^{\ast m}=f\ast
f\ast\ldots\ast f$ - $m$-times for a function (or a measure) $f$
and to avoid any misunderstandings we will not use pointwise
multiplication for functions up to the end of this chapter. Now,
we prove a lemma which is a simple corollary of the last
proposition.
\\
From now on, unless otherwise stated, $||\cdot||$ denotes the
$L^{1}(\mathbb{T})$ norm.
\begin{lem}\label{uzy}
Let $f$ be a trigonometric polynomial such that
\begin{equation*}
\widehat{f}(\mathbb{Z})=\{0,\lambda_{1},\lambda_{2},\ldots,\lambda_{k}\}.
\end{equation*}
Put $\lambda_{0}=0$ and define $\delta=\min_{i\neq
j}|\lambda_{i}-\lambda_{j}|$,
$\lambda_{max}=\max\{|\lambda|:\lambda\in\widehat{f}(\mathbb{Z})\}$.
Then, for every $m\in\mathbb{N}$
\begin{equation*}
||f^{m}||\leq k\delta^{-k}2^{k}||f||^{k}\lambda_{max}^{m}.
\end{equation*}
\end{lem}
\begin{proof}
We easily see that, if $f$ is a polynomial, its spectrum in the algebra
$M(\mathbb{T})$ is equal to $\widehat{f}(\mathbb{Z})$. Hence,
by Proposition $\ref{woj}$, we have
\begin{equation*}
f=\lambda_{1}f_{1}+\lambda_{2}f_{2}+\ldots+\lambda_{k}f_{k}
\end{equation*}
for some orthogonal idempotent polynomials $f_{k}$. Simple
calculation shows that
\begin{equation*}
f^{m}=\left(\sum_{l=1}^{k}\lambda_{l}f_{l}\right)^{m}=\sum_{l=1}^{k}\lambda_{l}^{m}f_{l}.
\end{equation*}
Applying the estimate from Proposition $\ref{woj}$ we complete the proof
\begin{equation*}
||f^{m}||\leq \sum_{l=1}^{k}|\lambda_{l}|^{m}||f_{l}||\leq
k\delta^{-k}2^{k}||f||^{k}\lambda_{max}^{m}.
\end{equation*}
\end{proof}
The next theorem allows the possibility to relax the strict spectral
conditions of the previous lemma.
\begin{tw}\label{trala}
Let $\Lambda\subset\mathbb{C}$ be a finite set with
$\#\Lambda=m+1$ such that $0\in\Lambda$ and let us denote
$\lambda_{max}=\max\{|\lambda|:\lambda\in\Lambda\}$. Then there
exists $C=C(\Lambda)$ such that, for every $K>0$ and every $k>m$
there exists $\varepsilon=\varepsilon(K,\Lambda)$ with the
property: if a trigonometric polynomial $f$ with $||f||\leq K$
satisfies
$\widehat{f}(\mathbb{Z})\subset\Lambda+B(0,\varepsilon)$, then the
following inequality holds
\begin{equation*}
||f^{k}||\leq C\lambda_{max}^{k-m}||f||^{m},
\end{equation*}
where $C=C(\Lambda)=C(t\cdot\Lambda)$ for any
$t\in\mathbb{C}\setminus\{0\}$.
\end{tw}
\begin{proof}
Let $\lambda_{0}=0$,
$\Lambda=\{0,\lambda_{1},\lambda_{2},\ldots,\lambda_{m}\}$ and put
$\delta=\min_{i\neq j}|\lambda_{i}-\lambda_{j}|$. We may assume
that $\varepsilon<\frac{\delta}{2}$ which guarantees
$B(\lambda_{i},\varepsilon)\cap
B(\lambda_{j},\varepsilon)=\emptyset$ for $i\neq j$. Let us take
$n\in\mathbb{Z}$, then there exists a unique $\lambda_{i_{n}}$ with the
property
$\min_{j=0,1,\ldots,m}|\lambda_{j}-\widehat{f}(n)|=|\lambda_{i_{n}}-\widehat{f}(n)|$.
Now, we define the polynomial $f_{0}$ by the condition
$\widehat{f_{0}}(n)=\lambda_{i_{n}}\in\Lambda$ for every
$n\in\mathbb{Z}$. It is obvious that
$\widehat{f_{0}}(\mathbb{Z})\subset\Lambda$. Moreover, $g=f-f_{0}$
satisfies the property $\widehat{g}(\mathbb{Z})\subset
B(0,\varepsilon)$. We have to estimate $||g||$ (an upper bound for
$||f_{0}||$ follows from Lemma $\ref{uzy}$). A simple observation
gives that, if $\widehat{f}(l)=0$ for some $l\in\mathbb{Z}$, then
$\widehat{g}(l)=0$. For any polynomial $h$, let us write
$\#h=\#\{n\in\mathbb{Z}:\widehat{h}(n)\neq 0\}$. Using Parseval's
identity we obtain
$||g||_{L^{1}(\mathbb{T})}\leq ||g||_{L^{2}(\mathbb{T})}\leq
\varepsilon\sqrt{\# g}\leq\varepsilon\sqrt{\#f}$.
Putting
$\gamma=\min_{n\in\mathbb{Z}}\{|\widehat{f}(n)|:\widehat{f}(n)\neq 0\}$
we obtain from the McGehee-Pigno-Smith + Konyagin theorem
$\gamma ||f||\geq L\ln(\# f)$
which leads to
$\sqrt{\#f}\leq\exp(\frac{\gamma}{2L}||f||)\leq\exp(\frac{\gamma}{2L}K)$.
Taking it all together we have
$||g||\leq\varepsilon\exp(\frac{\gamma}{2L}K)$.
Finally
$||f^{k}||\leq ||f_{0}^{k}||+\sum_{l=0}^{k-1}{k\choose
l}||f_{0}^{l}||||g||^{k-l}$.
Using Lemma $\ref{uzy}$ we obtain
\begin{equation*}
||f_{0}^{l}||\leq
m\delta^{-m}2^{m}\lambda_{\max}^{l}||f||^{m}\text{ for
$l=1,2\ldots,k$}
\end{equation*}
Moreover, we have
$||g||^{k-l}\leq \varepsilon^{k-l}\exp(\frac{\gamma(k-l)}{2l}K)$.
Collecting these facts, we get
\begin{equation*}
\begin{split}
||f^{k}||\leq
m\delta^{-m}2^{m}||f||^{m}(\lambda_{max}^{k}+\varepsilon\sum_{l=1}^{k-1}{k\choose
l}\lambda_{max}^{l}\varepsilon^{k-l-1}\exp(\frac{\gamma(k-l)}{2l}K)\\
+\frac{1}{||f||^{m}}\varepsilon^{k-1}\exp(\frac{m\gamma}{2l}K))
\end{split}
\end{equation*}
Taking $\varepsilon$ so small that the expression in parenthesis is
smaller than $2\lambda_{\max}^{k}$ we finally get
$||f^{k}||\leq m\delta^{-m}2^{m+1}||f||^{m}\lambda_{max}^{k}$.
Putting $C=m2^{m+1}\frac{\lambda_{max}^{m}}{\delta^{m}}$ we have
$C=C(\Lambda)=C(t\cdot\Lambda)$ for every
$t\in\mathbb{C}\setminus\{0\}$ which gives the desired estimation.
\end{proof}
We now introduce the following notation. Let $C>0$
and $k\in\mathbb{N}$, $k\geq 2$. We say that a compact set
$\Lambda\subset\mathbb{C}$ with $0\notin\Lambda$ belongs to the class $U(C,k)$ provided
for every $K>0$ there exists an open neighborhood $V_{K}$ of
$\Lambda$ such that for every $\mu\in M_{0}(\mathbb{T})$
satisfying $||\mu||\leq K$ and $\sigma(\mu)\subset V_{K}\cup\{0\}$ we have
\begin{equation*}
||\mu^{k}||_{M(\mathbb{T})}\leq C ||\mu||_{M(\mathbb{T})}^{k-1}.
\end{equation*}
By Theorem $\ref{trala}$, every finite set $\Lambda\subset\mathbb{C}$ with $0\notin\Lambda$ and $\#\Lambda=k$ belongs to the class $U(C,k)$.
\begin{tw}\label{zb}
Let $C>0$ and $k\in\mathbb{N}$. Assume that the sets $X,Y\in
U(C,k)$ are such that $X\subset B(0,r)$ and
$Y\subset\{z\in\mathbb{C}:|z|>R\}$ for some $R,r>0$. Then
for every $C'>C$, there exists $\varepsilon=\varepsilon(r,R,C')>0$
such that $\varepsilon X\cup Y\in U(C',k)$.
\end{tw}
\begin{proof}
Let us fix $K>0$ and take $\mu\in M_{0}(\mathbb{T})$ such that
$$
\sigma(\mu)\subset \left(\varepsilon X\cup Y+B(0,\delta)\right)\cup\{0\}
$$
 and
$||\mu||<K$. Since $\mu\in M_{0}(\mathbb{T})$, there are only
finitely many $n\in\mathbb{Z}$ satisfying $\widehat{\mu}(n)\in
Y+B(0,\delta)$. Hence, we may define the polynomial $f$ by the
conditions $\widehat{f}(n)=\widehat{\mu}(n)$, if
$\widehat{\mu}(n)\in Y+B(0,\delta)$ and $\widehat{f}(n)=0$
otherwise. Then the measure $\nu=\mu-f$ satisfies
$\widehat{\nu}(\mathbb{Z})\subset \left(\varepsilon X\cup
B(0,\delta)\right)\cup{0}$ and the equality $\mu=f+\nu$ holds.
Now, we apply Lemma $\ref{gl2}$. For sufficiently small
$\varepsilon=\varepsilon(c)$ we get $||f||_{L^{1}(\mathbb{T})}\leq
c||\mu||$, where $c$ is any number greater than $1$ and,
consequently $||\nu||\leq ||\mu||+||f||\leq (1+c)||\mu||$. The
measure $\varepsilon^{-1}\nu$ has its Fourier coefficients
contained in
$\left(X+B(0,\varepsilon^{-1}\delta)\right)\cup\{0\}$. By the
assumption there exists $\delta>0$ such that
$||(\varepsilon^{-1}\nu)^{k}||\leq C ||\varepsilon^{-1}\nu||^{k-1}$
which yields
$||\nu^{k}||\leq C\varepsilon ||\nu||^{k-1}$.
We also have $\widehat{f}(\mathbb{Z})=\sigma(f)\subset
\left(Y+B(0,\delta)\right)\cup\{0\}$ and $Y\in U(C,k)$. Hence, taking smaller $\delta$ if
necessary, we get
$||f^{k}||\leq C ||f||^{k-1}$.
Clearly, we may assume that the sets $\varepsilon X+B(0,\delta)$ and
$Y+B(0,\delta)$ are disjoint which leads to $\nu\ast f=0$.
Performing a simple calculation we obtain
\begin{equation*}
\begin{split}
||\mu^{k}||&=||f^{k}+\nu^{k}||\leq ||f^{k}||+||\nu^{k}||\leq
C||f||^{k-1}+C\varepsilon ||\nu||^{k-1}\\
&\leq C(c^{k-1}+\varepsilon(1+c)^{k-1})\cdot ||\mu||^{k-1}.
\end{split}
\end{equation*}
Since $c$ can be chosen arbitrarily close to $1$ and $\varepsilon$
may be as small as we wish, the theorem follows.
\end{proof}
In the formulation and the proof of next theorem we will write
$s(n)$ instead of $s_{n}$ for a clearer display.
\begin{tw}\label{naj}
Let $0<a<b$ and $k\in\mathbb{N}$, $k\geq 2$. For every $C>0$ there exist a sequence of continuous functions
$\psi_{l}:\mathbb{R}_{+}\mapsto\mathbb{R}_{+}$, $l\in\mathbb{N}$ such that whenever a
decreasing sequence $s(n)$ tending to zero satisfies
\begin{equation*}
\frac{s(2^{l}\cdot n+2^{l-1}+1)}{s(2^{l}\cdot
n+1)}<\psi_{l}\left(a\cdot\frac{s(2^{l}\cdot n+2^{l-1})}{s(2^{l}\cdot n+1)}\right)
\end{equation*}
for every $l\in\mathbb{N}$, $n=0,1,\ldots,2^{l-1}$ and $A_{n}\in U(C,k)$, $A_{n}\subset
B(0,b)\cap\{z\in\mathbb{C}:|z|>a\}$ and $(r_{n})$ tends to zero
rapidly enough, then any measure $\mu\in M_{0}(\mathbb{T})$ with
$||\mu||_{M(\mathbb{T})}\leq K$ and with property
\begin{equation*}
\widehat{\mu}(\mathbb{Z})\subset\bigcup(s(n)\cdot
A_{n}+B(0,r_{n}))\cup\{0\}
\end{equation*}
satisfies $\mu^{k}\in L^{1}(\mathbb{T})$.
\end{tw}
\begin{proof}
For simplicity assume $s(1)=1$. Let $C_{m}$ be an increasing sequence of real numbers such that $C<C_{m}<\widetilde{C}$ for all $m\in\mathbb{N}$ and some $\widetilde{C}>0$. Let us also define $\psi_{l}$ by the formula $\widetilde{\psi}_{l}(\cdot)=\varepsilon(b,\cdot,C_{l})$ where $\varepsilon(\cdot,\cdot,\cdot)$ is the same as in the previous theorem for $C:=C_{l-1}$. We show first by induction that
for every $m,n\geq 0$,
\begin{equation}\label{row1}
B_{m,n}=\bigcup_{j=n\cdot
2^{m}+1}^{(n+1)2^{m}}\frac{s(j)}{s(n\cdot 2^{m}+1)}\cdot A_{j}\in
U(C_{m},k),
\end{equation}
Indeed, for $m=0$ clearly $B_{0,n}=A_{n+1}$. For $m>0$ the
interval of integers $[n\cdot 2^{m}+1,(n+1)\cdot 2^{m}]$ is a
disjoint union of the intervals $[(2n)\cdot 2^{m-1}+1,(2n+1)\cdot
2^{m-1}]$ and $[(2n+1)\cdot 2^{m-1}+1,(2n+2)\cdot 2^{m-1}]$, which
implies
\begin{equation}\label{row2}
B_{m,n}=B_{m-1,2n}\cup\frac{s((2n+1)2^{m-1}+1)}{s(n\cdot
2^{m}+1)}\cdot B_{m-1,2n+1}.
\end{equation}
It is easy to check that
\begin{equation*}
B_{m-1,2n}\subset\{z\in\mathbb{C}:|z|>a\cdot\frac{s(n\cdot
2^{m}+2^{m-1})}{s(n\cdot 2^{m}+1)}\}
\end{equation*}
and
$B_{m-1,2n+1}\subset B(0,b)$.
From the previous theorem it follows that taking in ($\ref{row2})$ the
coefficient $\frac{s((2n+1)2^{m-1}+1)}{s(n\cdot 2^{m}+1)}$ such that
\begin{equation*}
\frac{s((2n+1)2^{m-1}+1)}{s(n\cdot 2^{m}+1)}\leq \widetilde{\psi}_{m}\left(a\cdot \frac{s(n\cdot2^{m}+2^{m-1})}{s(n\cdot 2^{m}+1)}\right)
\end{equation*}
the resulting union belongs to $U(C_{m},k)$ and hence to $U(\widetilde{C},k)$.
\\
Applying ($\ref{row1}$) to $n=1$ and $n=0$, we get that for every
$m=0,1,2,\ldots$
\begin{equation*}
B_{m-1,1}=\bigcup_{j=2^{m-1}+1}^{2^{m}}\frac{s(j)}{s(2^{m-1}+1)}\cdot
A_{j}\in U(\widetilde{C},k)
\end{equation*}
and
\begin{equation}\label{row3}
B_{m,0}\setminus B_{m-1,0}=\bigcup_{j=2^{m-1}+1}^{2^{m}}s(j)\cdot
A_{j}=s(2^{m-1}+1)\cdot B_{m-1,1}
\end{equation}
Then it follows from Theorem $\ref{zb}$ that there exists $r_{m}'$
such that for every measure $\gamma\in M(\mathbb{T})$ with
\begin{equation}\label{row4}
||\gamma||_{M(\mathbb{T})}\leq \frac{4K}{s(2^{m-1}+1)}\text{ and
}\widehat{\gamma}(\mathbb{Z})\subset B_{m-1,1}+B(0,r_{m}')
\end{equation}
there is
\begin{equation}{\label{row5}}
||\gamma^{k}||_{M(\mathbb{T})}< \widetilde{C}||\gamma||^{k-1}.
\end{equation}
Let $\mu\in M_{0}(\mathbb{T})$ with $||\mu||_{M(\mathbb{T})}\leq
K$ satisfy
\begin{equation*}
\widehat{\mu}(\mathbb{Z})\subset\bigcup_{m=1}^{\infty}\left((B_{m,0}\setminus
B_{m-1,0})+B(0,r_{m})\right)
\end{equation*}
where $r_{m}=r_{m}'\cdot s(2^{m-1}+1)$. Since $\mu\in
M_{0}(\mathbb{T})$, for any fixed $m\in\mathbb{N}$ there exist only
finitely many $p\in\mathbb{N}$ such that
$\widehat{\mu}(p)\in(B_{m,0}\setminus B_{m-1,0})+B(0,r_{m})$ (we may assume that these sets are disjoint for different $m's$). Hence
we are able to define the polynomials $f_{m}$ by the condition
$\widehat{f_{m}}(p)=\widehat{\mu}(p)$ for $p\in\mathbb{Z}$ such
that $\widehat{\mu}(p)\in(B_{m,0}\setminus B_{m-1,0})+B(0,r_{m})$
and $\widehat{f_{m}}(p)=0$ for other $p's$. Then we have
$\widehat{\mu}=\sum_{m=0}^{\infty}\widehat{f_{m}}$.
The polynomials
$f_{m}'=\frac{f_{m}}{s(2^{m-1}+1)}$
satisfy ($\ref{row4})$. In fact, the second part of
($\ref{row4})$ follows directly from ($\ref{row3}$). For the first
part it is enough to show that $||f_{m}||\leq 4K$. This follows
from the observation that
\begin{equation*}
f_{m}=\left(\sum_{j=0}^{m}f_{j}\right)-\left(\sum_{j=0}^{m-1}f_{j}\right),
\end{equation*}
the triangle inequality, Lemma $\ref{gl2}$ and properties
\begin{equation*}
\sum_{j=0}^{m}f_{j}\in G(a\cdot s(2^{m}))\text{ and
}\sum_{j=m+1}^{\infty}f_{j}\in F(b\cdot s(2^{m}+1)).
\end{equation*}
Indeed, let us put in Lemma $\ref{gl2}$
\begin{equation*}
f=\sum_{j=0}^{m}f_{j},\text{  }\nu=\mu-f\text{  and
}\varepsilon=1
\end{equation*}
and define a sequence of functions $\psi_{m}$ by the formula
\begin{equation*}
\psi_{m}(\cdot)=\min\left(\widetilde{\psi}_{m}(\cdot),\frac{1}{b}\delta(1,\cdot,K)\right).
\end{equation*}
Then the assumptions of Lemma $\ref{gl2}$ are fulfilled which leads to $||f||<2K$. The second
term is estimated analogously.

Finally we have, by ($\ref{row5}$)
$||f_{m}^{k}||_{L^{1}(\mathbb{T})}<s(2^{m-1}+1)\widetilde{C}||f_{m}||^{k-1}_{L^{1}(\mathbb{T})}<s(2^{m-1}+1)
\widetilde{C}(4K)^{k-1}$.
It leads to the conclusion that the series
$\sum_{m=0}^{\infty}f_{m}^{k}$
is absolutely convergent in $L^{1}(\mathbb{T})$ to $\mu^{k}$ which
finishes the proof.
\end{proof}

\section{Reduction to the case of $M_{0}(\mathbb{T})$}

The first result which will be used in this section is the
following theorem taken from the  book $[GM]$, closely related to
results from $[DK]$.

\begin{tw}\label{mc}
Let $r\in\mathbb{N}$, $r\geq 2$ and $\mu\in M_{c}(\mathbb{T})$. Let us define set $Q=Q(\mu)$
\begin{equation*}
Q=\{n\in\mathbb{Z}:|\widehat{\mu}(n)|\geq 1\}
\end{equation*}
and suppose that $|\widehat{\mu}(n)|\leq e^{-r}$ for $n\notin Q$.
\begin{equation*}
\begin{split}
&\text{If }||\mu||<\frac{r^{\frac{1}{2}}}{4},\text{ then }Q\text{ is a finite set}.
\\&\text{If }||\mu||<\frac{r^{\frac{1}{2}}}{4}\text{ and }N\in\mathbb{N}\text{ is such that } r\leq\left(\ln\frac{N}{4}\ln\ln N\right)^{\frac{1}{2}}\text{ then }\#Q<N.
\end{split}
\end{equation*}
\end{tw}

Applying the above theorem we prove that continuous measures
belong to $M_{0}(\mathbb{T})$ if special
assumptions are imposed on the range of their Fourier transforms. In the next lemma
we denote $L(r,t)=\{z\in\mathbb{C}:r<|z|<t\}$ ($0<r<t$).
\begin{lem}\label{dozera}
Let $w_{k}$, $t_{k}$, $w_{k}<t_{k}$ be sequences of positive real
numbers such that
\begin{enumerate}
    \item $t_{k}\rightarrow 0$ as $k\rightarrow\infty$.
    \item Sequence $t_{k}\sqrt{\ln\frac{t_{k}}{w_{k}}}$ is
    increasing and divergent to $\infty$.
\end{enumerate}
Let $L_{k}=L(w_{k},t_{k})$. If $\mu\in M_{c}(\mathbb{T})$ satisfies
$\widehat{\mu}(\mathbb{Z})\cap L_{k}=\emptyset$ for
all  $k\in\mathbb{N}$, then $\mu\in
M_{0}(\mathbb{T})$.
\end{lem}
\begin{proof}
Let us define a sequence $x_{k}$ by formula
$x_{k}=\ln\frac{t_{k}}{w_{k}}$. Then, the second condition on
$w_{k}$, $t_{k}$ implies that the sequence $t_{k}\sqrt{x_{k}}$ is
increasing and divergent to $\infty$. Hence, there exists
$k_{0}\in\mathbb{N}$ such that for every $k>k_{0}$
\begin{equation*}
||\mu||<\frac{t_{k}\sqrt{x_{k}}}{4}.
\end{equation*}
By the assumptions of our theorem $\widehat{\mu}(\mathbb{Z})\cap
L_{k}=\emptyset$ for every $k>k_{0}$. That means
\begin{equation*}
\frac{\widehat{\mu}(m)}{{t_{k}}}\notin L(e^{-x_{k}},1)\text{ for
all $m\in\mathbb{Z}$, $k>k_{0}$}.
\end{equation*}
From Theorem $\ref{mc}$ we obtain that the set
\begin{equation*}
A_{k}=\{n\in\mathbb{Z}:\left|\frac{\widehat{\mu}(n)}{t_{k}}\right|>1\}=\{n\in\mathbb{Z}:\left|{\widehat{\mu}(n)}\right|>t_{k}\}
\end{equation*}
is finite for all $k>k_{0}$ which finishes the proof.
\end{proof}
\begin{rem}
Obviously the assumption $\widehat{\mu}(\mathbb{Z})\cap L_{k}=\emptyset$ from the preceding lemma can be replaced by the requirement $\widehat{\mu}(\mathbb{Z})\cap L_{k_{n}}=\emptyset$ for any subsequence $k_{n}$.
\end{rem}
This  will be used in the proof of Theorem $1$ in the next section.
The second reduction, needed in the proof of Theorem $2$, is as follows: having arbitrary $\mu\in
M(\mathbb{T})$ we split it in a standard way $\mu=\mu_{c}+\mu_{d}$
where $\mu_{c}$ is the continuous part of the measure and $\mu_{d}$
its discrete part. Then from assumptions on the set
$\widehat{\mu}(\mathbb{Z})$ we would like to extract information
about set $\widehat{\mu_{c}}(\mathbb{Z})$ which with the aid of the last
lemma leads to the conclusion that $\mu_{c}\in M_{0}(\mathbb{T})$ and for measures from this class we shall
apply Theorem $\ref{naj}$. One thing remains to be proved - if
$\mu\in M_{0}(\mathbb{T})$ has a natural spectrum and $\nu$ is an arbitrary
measure with a natural spectrum, then $\mu+\nu$ also has a
natural spectrum. The key to obtain this fact is an important
theorem of Zafran (see [Z1]). To formulate it we introduce the
following definition.

\begin{de}
Let $\mathscr{C}$ denote the set of measures with a natural spectrum
with Fourier-Stieltjes coefficients from $c_0$, i.e.
\begin{equation*}
\mathscr{C}=\{\mu\in
M_{0}(\mathbb{T}):\sigma(\mu)=\overline{\widehat{\mu}(\mathbb{Z})}=\widehat{\mu}(\mathbb{Z})\cup\{0\}\}.
\end{equation*}
\end{de}
For any commutative Banach algebra $A$ we denote by
$\mathfrak{M}(A)$ the space of the maximal modular ideals of $A$
identified also as the set of all multplicative linear functionals
on $A$ (cf. [Ż]).
\begin{tw}[Zafran] \label{zaf} The following hold true:
\begin{enumerate}
    \item if
    $h\in\mathfrak{M}(M_{0}(\mathbb{T}))\setminus\mathbb{Z}$, then
    $h(\mu)=0$ for $\mu\in\mathscr{C}$.
    \item $\mathscr{C}$ is closed ideal in
    $M_{0}(\mathbb{T})$.
    \item $\mathfrak{M}(\mathscr{C})=\mathbb{Z}$.
\end{enumerate}
\end{tw}
It easy an elementary fact that
$L^{1}(\mathbb{T})\subset\mathscr{C}$ and from preceding theorem
we conclude
\begin{equation*}
\mathrm{Rad}(L^{1}(\mathbb{T}))=\{\mu\in
M(\mathbb{T}):\exists_{k\in\mathbb{N}}\mu^{\ast k}\in
L^{1}(\mathbb{T})\}\subset\mathscr{C}.
\end{equation*}
On the contrary to the second condition of Zafran's theorem, the
sum of two measures with a natural spectrum does not necessarily
have a natural spectrum. The proof of this fact is based on the
construction of the measure supported on an independent Cantor set
as in [R] (see also [Z1] for details). However, as stated before,
assuming more on one summand provides the desired property.

\begin{tw}\label{za}
The sum of two measures with natural spectrum has natural
spectrum, if one of them has Fourier coefficients tending to zero.
\end{tw}
\begin{proof}
The spectrum of a measure is the image of its Gelfand transform.
Hence using the result of Zafran for $\mu\in\mathscr{C}$ and $\nu$
with natural spectrum, we obtain
\begin{equation*}
\aligned
\sigma(\mu+\nu)=&\{\varphi(\mu+\nu):\varphi\in\mathfrak{M}(M(\mathbb{T}))\}\\
=&(\widehat{\mu}+\widehat{\nu})(\mathbb{Z})\cup\{\varphi(\nu):\varphi\in\mathfrak{M}(M(\mathbb{T}))\setminus\mathbb{Z}\}.
\endaligned
\end{equation*}
Since $\nu$ has a natural spectrum, for every
$\varphi\in\mathfrak{M}(M(\mathbb{T}))$ we have
$\varphi(\nu)\in\overline{\widehat{\nu}(\mathbb{Z})}$. Now we
consider two cases
\begin{enumerate}
\item
$\varphi(\nu)\in\overline{\widehat{\nu}(\mathbb{Z})}\setminus\widehat{\nu}(\mathbb{Z})$.
\item $\varphi(\nu)\in\widehat{\nu}(\mathbb{Z})$.
\end{enumerate}
In the first case there exists an increasing sequence
$(n_{k})_{k\in\mathbb{N}}$ of integers such that
$\lim_{k\rightarrow\infty}\widehat{\nu}(n_{k})=\varphi(\nu)$.
Since $\mu\in M_{0}(\mathbb{T})$, we get
$\varphi(\nu)=\lim_{k\rightarrow\infty}\widehat{\nu}(n_{k})=\lim_{k\rightarrow\infty}(\widehat{\mu}
+\widehat{\nu})(n_{k})$.
Hence $\varphi(\nu)\in\overline{(\widehat{\mu}+\widehat{\nu})(\mathbb{Z})}$,
which completes the proof in this case.
\\
In the second case $\varphi(\nu)=\widehat{\nu}(n_{0})$ for some
$n_{0}\in\mathbb{Z}$. If $\widehat{\mu}(n_{0})=0$ then
$\varphi(\nu)=\widehat{\nu}(n_{0})+\widehat{\mu}(n_{0})\in(\widehat{\mu}+\widehat{\nu})(\mathbb{Z})$
and the result follows. Assume now $\widehat{\mu}(n_{0})\neq 0$.
If $\widehat{\nu}(n_{0})$ is an accumulation point of
$\sigma(\nu)=\overline{\widehat{\nu}(\mathbb{Z})}$, then we
proceed as in the first case. It remains to consider the case when
$\widehat{\nu}(n_{0})$ is an isolated point of
$\overline{\widehat{\nu}(\mathbb{Z})}$. Because $\mu\in
M_{0}(\mathbb{T})$ we get that also $\widehat{\nu}(n_{0})$ is an
isolated point of
$\overline{(\widehat{\mu}+\widehat{\nu})(\mathbb{Z})}$. We will
prove a stronger statement, that $\widehat{\nu}(n_{0})$ is an
isolated point of $\sigma(\mu+\nu)$. Indeed, suppose on the
contrary that there exists a sequence of complex numbers
$(\lambda_{k})_{k\in\mathbb{N}}\subset\sigma(\mu+\nu)$ tending to
$\widehat{\nu}(n_{0})$. Since the spectrum of a measure is the
image of its Gelfand transform, we can choose a sequence
$\varphi\notin
(h_{k})_{k\in\mathbb{N}}\in\mathfrak{M}(M(\mathbb{T}))$ such that
$h_{k}(\mu+\nu)=\lambda_{k}$. Without losing generality, we may
assume that for a sufficiently large $k$, the functionals $h_{k}$
are not the Fourier coefficients (otherwise
$\widehat{\nu}(n_{0})\in\overline{(\widehat{\mu}+\widehat{\nu})(\mathbb{Z})}$
and the proof is finished). Using again the theorem of Zafran we
get
\begin{equation*}
\lim_{k\rightarrow\infty}h_{k}(\mu+\nu)=\lim_{k\rightarrow\infty}h_{k}(\nu)=\widehat{\nu}(n_{0}).
\end{equation*}
But
$h_{k}(\nu)\in\sigma(\nu)=\overline{\widehat{\nu}(\mathbb{Z})}$.
Hence $\widehat{\nu}(n_{0})$ is not an isolated point of
$\overline{\widehat{\nu}(\mathbb{Z})}$, which contradicts the
assumption.
\\
Since $\sigma(\mu+\nu)$ has a complex number
$\widehat{\nu}(n_{0})$ as an isolated point, we can find two open
sets $A,B\subset\mathbb{C}$ such that $A\cap B=\emptyset$,
$\sigma(\mu+\nu)\subset A\cup B$, $\widehat{\nu}(n_{0})\in B$ and
$\sigma(\mu+\nu)\setminus\widehat{\nu}(n_{0})\subset A$. Let $f$
be a holomorphic function defined on $A\cup B$ by putting $f(z)=z$
for $z\in A$ and $f\equiv \widehat{\nu}(n_{0})+1$ on $B$. By the
spectral mapping theorem there exists a measure
$\upsilon:=f(\mu+\nu)\in M(\mathbb{T})$ satisfying
\begin{equation*}
\widehat{\upsilon}(m)=f((\widehat{\mu}+\widehat{\nu})(m))\text{ for all
$m\in\mathbb{Z}$}.
\end{equation*}
By the definition of $f$ we have
$\widehat{\upsilon}(n)=(\widehat{\mu}+\widehat{\nu})(n)$ for
$n\neq n_{0}$. Moreover, since we have assumed
$\widehat{\mu}(n_{0})\neq 0$ we have
$\widehat{\mu}(n_{0})+\widehat{\nu}(n_{0})\neq
\widehat{\nu}(n_{0})$ and
$\widehat{\mu}(n_{0})+\widehat{\nu}(n_{0})\in A$ which leads to
\begin{equation*}
\widehat{\upsilon}(n_{0})=(\widehat{\mu}+\widehat{\nu})(n_{0}).
\end{equation*}
Therefore, the measures $\upsilon=f(\mu+\nu)$ and $\mu+\nu$ have
the same Fourier coefficients, which implies $f(\mu+\nu)=\mu+\nu$.
From point 1. of Zafran's theorem we have $\varphi(\mu)=0$ which
by the properties of functional calculus implies
\begin{gather*}
\widehat{\nu}(n_{0})=\varphi(\mu+\nu)=\varphi(f(\mu+\nu))=f(\varphi(\mu+\nu))=\\
=f(\varphi(\nu))=f(\widehat{\nu}(n))=\widehat{\nu}(n)+1.
\end{gather*}
This contradiction completes the proof.
\end{proof}

\section{Proofs of the Main Theorems}
We begin with the proof of Theorem $1$. Let sequences $t_{k}$,
$w_{k}$ and $L_{k}$ be as in Lemma $\ref{dozera}$ and let $b>2$,
$a<1$. We put $A_{n}=A=\{-1,1\}$ for $n\in\mathbb{N}$ and take $C>0$ from Theorem $12$ such that $A\in U(C,2)$. Next, let us denote by $\psi_{n}$ the sequence of functions from Theorem
$\ref{naj}$. We construct inductively sequence $\varepsilon_{n}$
as follows: $\varepsilon_{0}=1$ and if
$\varepsilon_{0},\ldots,\varepsilon_{n}$ are chosen let us take
the smallest $k$ such that
$t_{k}<\varepsilon_{1}\cdot\ldots\cdot\varepsilon_{n}$ and rename
it as $k_{n}$. Now we put  $\varepsilon_{n+1}$ any number which
satisfies $0<\varepsilon_{n+1}<\frac{1}{2}w_{k_{n}}$ and
$\varepsilon_{n+1}<\psi_{n}(a\varepsilon_{1}\cdot\ldots\cdot\varepsilon_{n})$.
Every $n\in\mathbb{N}$ has a unique binary expansion
\begin{equation*}
n=\sum a_{i}2^{i},\text{ $a_{i}\in\{0,1\}$}.
\end{equation*}
We put $s(n+1)=\prod\varepsilon_{i}^{a_{i}}$, where the $a_{i}$ are the
coefficient of the binary expansion given above. We choose the sequence $r(n)$
from Theorem $\ref{naj}$ and modify it (if necessary) to guarantee
the conditions: $r(2^{n-1})<t_{k_{n}}$ and
$r(2^{n})<\frac{1}{2}w_{k_{n}}$. Finally, we put
\begin{equation*}
U=\bigcup_{n\in\mathbb{N}}\left(s(n)\cdot A_{n}+B(0,r(n))\right).
\end{equation*}
Let $\mu\in M_{c}(\mathbb{T})$ be a measure such that
$\widehat{\mu}(\mathbb{Z})\subset U\cup\{0\}$. Then, by the
construction $\left(U\cup\{0\}\right)\cap I_{k_{n}}=\emptyset$. By
Lemma $\ref{dozera}$ and the remark following it, $\mu\in
M_{0}(\mathbb{T})$. Finally, Theorem $\ref{naj}$ yields that
$\mu^{2}\in L^{1}(\mathbb{T})$. Hence, by Zafran's theorem
$\mu\in\mathscr{C}$, i.e. $\mu$ has a natural spectrum.

We move on now to the proof of Theorem 2. We start from the
following simple lemma whose proof is left to the reader.
\begin{lem}\label{cz}
Let $S=\bigcup_{k=1}^\infty B_k\subset\mathbb{C}$ be the union of balls such that
$0\in\overline{S}$. Then there exists a (topological) Cantor set $K$  such that
$K-K\subset\overline{S\cup -S}$.
\end{lem}
Let $D=\{-1,1\}$ and $E=\{-2,-1,1,2\}$. By Theorem $\ref{trala}$,
$D\in U(C,2)$ and $E\in U(C,4)$ for some $C>0$. Let
$(s_n)_{n=1}^\infty:=(s(n))_{n=0}^{\infty}$ and
$(r_n)_{n=1}^\infty$ satisfy the conditions of Theorem $\ref{naj}$
for $A_n=D$, $n=1,2,\dots$ and
$(s_n)_{n=1}^\infty:=(s(n))_{n=1}^{\infty}$ and
$(r'_n)_{n=1}^\infty$ satisfy the conditions of Theorem
$\ref{naj}$ for $A_n=E$, $n=1,2,\dots$  and, moreover,
$2|s_m|+r'_m+r'_n<r_n$ for $m>n$. Let $G_n$, $n=1,2,\dots$, be a
Cantor set satisfying, by Lemma \ref{cz},
\begin{equation*}
G_n-G_n\subset \overline{\bigcup_{k>n} B(s_k,r'_k)\cup\bigcup_{k>n}B(-s_{k},r_{k}')}.
\end{equation*}
We also put $A_{0}=B_{0}=G_{0}=\{0\}$, $r_{0}=r_{0}'=0$, $s_{0}=1$ and with a little abuse of notation $B(0,0)=\{0\}$. Then we define
\begin{equation*}
X=\bigcup_{n=0}^\infty (s_nA_n+G_n).
\end{equation*}
Clearly the set $X$ is closed. Suppose now that $\widehat{\mu}(\mathbb{Z})\subset X$. We will now use the following result from $[GW]$.
\begin{lem}
Let $\mu\in M(\mathbb{T})$. Then $\overline{\widehat{\mu_{d}}(\mathbb{Z})}\subset\overline{\widehat{\mu}(\mathbb{Z})}$ where $\mu_{d}$ is the discrete part of measure $\mu$.
\end{lem}
Therefore
\begin{equation*}
\aligned
\widehat{\mu}_c(\mathbb{Z})&\subset \widehat{\mu}(\mathbb{Z})-
\widehat{\mu}_d(\mathbb{Z})\\&\subset X-X\\
&\subset\big(\bigcup_{n=0}^\infty (s_nA_n+G_n)\big)-\big(\bigcup_{n=0}^\infty (s_nA_n+G_n)\big)\\
&\subset\bigcup_{n\neq k} (s_nA_n-s_kA_k+G_n-G_k)\cup\bigcup_n((s_nA_n+G_n)-(s_nA_n+G_n))\\
&\subset\bigcup_{n<k} (s_nA_n+B(0,2|s_k|+r'_n+r'_k))\\
&\cup\bigcup_{n}(s_nB_n+G_n-G_n)\cup\bigcup_{n} (G_n-G_n)\\
&\subset\bigcup_{n<k} (s_nB_n+B(0,2|s_k|+r'_n+r'_k))\\
&\qquad\qquad\cup\quad\bigcup_{n}(s_nB_n+B(0,2r'_n))\cup\bigcup_{n}(s_nB_n+B(0,r'_n))\\
&\subset\bigcup_{n}(s_nB_n+B(0,r_n))
\endaligned
\end{equation*}

By Theorem 1, we derive that $\mu_c$ has a natural spectrum.
Finally, since additionally $\mu_c\in M_0$, Theorem $\ref{za}$
yields that $\mu=\mu_c+\mu_d$ has a natural spectrum.

\section{The Example}

In this section we construct the example of a singular measure satisfying the assumptions of Theorem $1$ which we call the \textbf{product of Riesz-Rudin-Shapiro}. Our construction is an instance of \textit{generalized Riesz products}, which are elaborated in Chapter 5 of the book [HMP]. At the beginning let us recall a few results concerning the usual Riesz products. They are continuous, probabilistic measures on the circle group given as the weak-star limit of trigonometric polynomials of the form
\begin{equation*}
\prod_{k=1}^{N}(1+a_{k}\cos(n_{k}t)),
\end{equation*}
where $-1\leq a_{k}\leq 1$ and $n_{k}$ is a sequence of natural numbers satisfying the lacunary condition $\frac{n_{k+1}}{n_{k}}\geq 3$. We will write
\begin{equation*}
R(a_{k},n_{k})=\prod_{k=1}^{\infty}(1+a_{k}\cos(n_{k}t))
\end{equation*}
for the Riesz product built on the sequences $a_{k}$ and $n_{k}$ satisfying the above conditions. One of the oldest results (see [Z2]) on Riesz products is that $R(1,3^{k})$ is singular with respect to Lebesgue measure (we will write simply "singular" for this situation). However, a much more general theorem was proved in [BM]. In this formulation $\mu\bot\nu$ means that the measures $\mu,\nu\in M(\mathbb{T})$ are mutually singular and $\mu\sim\nu$ denotes the equivalence of measures i.e., $\mu$ is absolutely continuous with respect to $\nu$ and vice versa.
\begin{tw}[Brown,Moran]
If $a_{k}$, $b_{k}$ satisfies $-1\leq a_{k}, b_{k}\leq 1$ and the sequence of natural numbers $n_{k}$ has the property $\frac{n_{k+1}}{n_{k}}\geq 3$ then
\begin{equation*}
\begin{split}
R(a_{k},n_{k})\bot R(b_{k},n_{k})\Longleftrightarrow\sum_{k=1}^{\infty}(a_{k}-b_{k})^{2}=\infty,\\
R(a_{k},n_{k})\sim R(b_{k},n_{k})\Longleftrightarrow\sum_{k=1}^{\infty}(a_{k}-b_{k})^{2}<\infty.
\end{split}
\end{equation*}
\end{tw}
As we stated in the introduction, Riesz products may be used for a simple proof of the Wiener-Pitt phenomenon (see [G]). Moreover, Zafran, in his paper, gave a necessary and sufficient condition for the Riesz product $R(a_{k},n_{k})$ to have a natural spectrum under the assumption that the sequence $a_{k}$ converges to zero.
\begin{tw}[Zafran]
Let $a_{k}$ be a sequence tending to zero such that $-1\leq a_{k}\leq 1$ and $n_{k}$ be a sequence of natural numbers such that $\frac{n_{k+1}}{n_{k}}\geq 3$. Then the Riesz product $R(a_{k},n_{k})$ has a natural spectrum if, and only if, there exists $m\in\mathbb{N}$ such that
\begin{equation*}
\sum_{k=1}^{\infty}|a_{k}|^{m}<\infty.
\end{equation*}
\end{tw}
It is also proven in [BBM] that in the case when the Riesz product has all powers mutually singular, its spectrum is the whole disc $\{z\in\mathbb{C}:|z|\leq 1\}$. However, usual Riesz products are not sufficient for our needs and we move on to the construction of more general class of measures.
\\
Let us start with some preliminary lemmas and notation. The first one is a very simple arithmetic argument needed in calculating Fourier-Stieltjes coefficients of our measure.
\begin{lem}\label{aryt}
Let $\{m_{j}\}_{j=1}^{\infty}$, $\{r_{j}\}_{j=1}^{\infty}$ and $\{n_{j}\}_{j=1}^{\infty}$ be increasing sequences of positive integers with the property
\begin{equation*}
r_{k}>2\sum_{j=1}^{k-1}((2^{n_{j}}-1)m_{j}+r_{j})\text{ for }k\geq 2.
\end{equation*}
Moreover, let $\{c_{j}\}_{j=1}^{\infty}$ be a sequence of positive integers satisfying for $j\in\mathbb{N}$ condition
\begin{equation*}
c_{j}\in\{r_{j},m_{j}+r_{j},2m_{j}+r_{j},\ldots,(2^{n_{j}}-1)m_{j}+r_{j}\}.
\end{equation*}
Assume that an integer $s$ is expressible in the form
\begin{equation*}
s=\sum_{j=1}^{N}b_{j}c_{j}\text{ where }N\in\mathbb{N},\text{ }b_{j}\in\{-1,0,1\}\text{ and }b_{N}\neq 0.
\end{equation*}
Then this expression is unique.
\end{lem}
The proof is obvious and we omit it.
\\
The fundamental ingredient in our construction are the Rudin-Shapiro polynomials
(cf. [R]). We recall them in the next definition.
\begin{de}
Let $P_{0}\equiv 1$ and $Q_{0}\equiv 1$. We define inductively two sequences of polynomials by the formula
\begin{equation*}
\begin{split}
P_{n+1}(t)=P_{n}(t)+e^{i2^{n}t}Q_{n}(t),\\
Q_{n+1}(t)=P_{n}(t)-e^{i2^{n}t}Q_{n}(t).
\end{split}
\end{equation*}
\end{de}
We will reserve the name 'Rudin-Shapiro polynomials' for the sequence $\{P_{n}\}_{n=0}^{\infty}$. Now, we collect the well-known properties of these polynomials.
\begin{prop}\label{rs}
For every $n\in\mathbb{N}$ we have
\begin{equation*}
P_{n}(t)=\sum_{k=0}^{2^{n}-1}a_{k}e^{ikt}\text{ where }a_{k}\in\{-1,1\}\text{ for }k\in\{0,\ldots,2^{n}-1\}.
\end{equation*}
Hence $||P_{n}||_{L^{2}(\mathbb{T})}=2^{\frac{n}{2}}$. Also, $||P_{n}||_{C(\mathbb{T})}\leq 2^{\frac{n+1}{2}}$.
\end{prop}
Using the sequence $\{P_{n}\}_{n=1}^{\infty}$ we define another sequence  $\{w_{k}\}_{k=1}^{\infty}$ of polynomials.
\begin{de}\label{nw}
Let $\{P_{n}\}_{n=1}^{\infty}$ be the sequence of Rudin-Shapiro polynomials and $\{r_{k}\}_{k=1}^{\infty}$, $\{m_{k}\}_{k=1}^{\infty}$, $\{n_{k}\}_{k=1}^{\infty}$ be increasing sequences of positive integers. Let $\{\varepsilon_{k}\}_{k=1}^{\infty}$ be a decreasing sequence of positive numbers vanishing at infinity. We define polynomials $\{w_{k}\}_{k=1}^{\infty}$ by the formula
\begin{equation*}
w_{k}=\varepsilon_{k}P_{n_{k}}(m_{k}t)e^{ir_{k}t}+\varepsilon_{k}\overline{P_{n_{k}}(m_{k}t)}e^{-ir_{k}t}.
\end{equation*}
\end{de}
We summarize the properties of polynomials $w_{k}$.
\begin{prop}\label{postpr}
The polynomials $w_{k}$ are real-valued on $\mathbb{T}$ and have
the following form (a series with $'$ denotes a sum without the
term corresponding to $l=0$)
\begin{equation}\label{post}
w_{k}(t)=\varepsilon_{k}\sum_{l=-2^{n_{k}}+1}^{2^{n_{k}}-1'}a_{|l|}e^{it(lm_{k}+\mathrm{sgn}(l)r_{k})}+a_{0}(e^{-itr_{k}}+e^{itr_{k}})\text{
where }a_{l}\in\{-1,1\}.
\end{equation}
Hence,
$||w_{k}||^{2}_{L^{2}(\mathbb{T})}=\varepsilon^{2}_{k}(2^{n_{k}+1}-1)$.
Moreover, $||w_{k}||_{C(\mathbb{T})}\leq
\varepsilon_{k}2^{\frac{n_{k}+3}{2}}$.
\end{prop}
\begin{proof}
The polynomials $w_{k}$ are real-valued by Definition $\ref{nw}$.
Equation $(\ref{post})$ is straightforward by Proposition
$\ref{rs}$. The latter properties follows by $(\ref{post})$
and Proposition $\ref{rs}$.
\end{proof}
Equation $(\ref{post})$ exposes an important feature of polynomials $w_{k}$, namely the sequence of their Fourier coefficients has gaps with lengths equal to $m_{k}$.
\\
We are now ready to construct the Riesz-Rudin-Shapiro products. Proof of the following proposition is the standard argument based on fact that the weak-star convergence of a bounded sequence of measures follows from the pointwise convergence of its Fourier transforms (see for example [HMP]).
\begin{prop}\label{kons}
Let $\{r_{k}\}_{k=1}^{\infty}$, $\{m_{k}\}_{k=1}^{\infty}$, $\{n_{k}\}_{k=1}^{\infty}$ be increasing sequences of positive integers. Let $\{\varepsilon_{k}\}_{k=1}^{\infty}$ be a decreasing sequence of positive numbers vanishing at infinity and $\{w_{k}\}_{k=1}^{\infty}$ be the sequence of polynomials corresponding to them. Assume that the condition
\begin{equation}\label{lak}
r_{k}>2\sum_{j=1}^{k-1}((2^{n_{j}}-1)m_{j}+r_{j})\text{ for }k\geq 2
\end{equation}
is satisfied and moreover, for all $k\in\mathbb{N}$ we have $\varepsilon_{k}2^{\frac{n_{k}+3}{2}}<1$. Then the sequence of polynomials
\begin{equation*}
f_{N}(t)=\prod_{k=1}^{N}(1-w_{k}(t)).
\end{equation*}
converges in the weak-star topology of $M(\mathbb{T})$ as $N\rightarrow\infty$ to some positive measure $\mu\in M_{0}(\mathbb{T})$ with $||\mu||_{M(\mathbb{T})}=1$ with the additional property
\begin{equation*}
\widehat{\mu}(\mathbb{Z})\subset\{\pm\prod_{k=1}^{m}\varepsilon_{k}^{l_{k}}:l_{k}\in\{0,1\},m\in\mathbb{N}\}\cup\{0\}
\end{equation*}
\end{prop}
We will  write
\begin{equation*}
\mu=\prod_{k=1}^{\infty}(1-w_{k})
\end{equation*}
to denote the measure $\mu$ obtained by the procedure
described above. Adding more restrictions to our sequences we get
the following
\begin{prop}\label{ly}
Let $\{r_{k}\}_{k=1}^{\infty}$, $\{m_{k}\}_{k=1}^{\infty}$, $\{n_{k}\}_{k=1}^{\infty}$ be increasing sequences of positive integers. Let $\{\varepsilon_{k}\}_{k=1}^{\infty}$ be a decreasing sequence of positive numbers vanishing at infinity and $\{w_{k}\}_{k=1}^{\infty}$ be the sequence of polynomials corresponding to them. Assume that the conditions
\begin{equation}
\begin{split}\label{zal}
r_{k}>2\sum_{j=1}^{k-1}((2^{n_{j}}-1)m_{j}+r_{j})\text{ for }k\geq 2,\\
m_{k}>2\sum_{j=1}^{k-1}((2^{n_{j}}-1)m_{j}+r_{j})\text{ for }k\geq 2
\end{split}
\end{equation}
are satisfied and moreover, for all $k\in\mathbb{N}$ we have $\varepsilon_{k}2^{\frac{n_{k}+3}{2}}<1$. Also, assume that there exists a constant $c>0$ such that
\begin{equation}\label{zal1}
\varepsilon_{k}^{2}(2^{n_{k}+1}-1)>c\text{ for all
$k\in\mathbb{N}$}.
\end{equation}
Then the measure
\begin{equation*}
\mu=\prod_{k=1}^{\infty}(1-w_{k})
\end{equation*}
does not belong to $L^{2}(\mathbb{T})$.
\end{prop}
\begin{proof}
The above assumptions  are sufficient to prove the existence of a
measure $\mu$ as announced in Proposition $\ref{kons}$. Let us
fix $N\in\mathbb{N}$ and consider the polynomial
\begin{equation*}
f_{N}(t)=\prod_{k=1}^{N}(1-w_{k}(t)).
\end{equation*}
An easy application of Parseval's identity gives
\begin{equation*}
\sum_{k=-\infty}^{\infty}|\widehat{\mu}(k)|^{2}\geq\sum_{k=-\infty}^{\infty}|\widehat{f_{N}}(k)|^{2}=||f_{N}||_{L^{2}(\mathbb{T})}^{2}.
\end{equation*}
Hence it is enough to show that  $||f_{N}||_{L^{2}(\mathbb{T})}^{2}\rightarrow\infty$ as $N\rightarrow\infty$. We proceed with the following calculation (we use the normalized Lebesgue measure on interval $[0,2\pi]$).
\begin{equation*}
||f_{N}||_{L^{2}(\mathbb{T})}^{2}=\int\prod_{k=1}^{N}(1-w_{k}(t))^{2}dt=\int\prod_{k=1}^{N}(1-2w_{k}(t)+w^{2}_{k}(t))dt.
\end{equation*}
Expanding the last product we get the terms which are the multiples of the expressions
\begin{equation*}
\pm\int w_{i_{1}}^{l_{1}}(t)\cdot w_{i_{2}}^{l_{2}}(t)\cdot\ldots\cdot w_{i_{m}}^{l_{m}}(t)dt=\pm\int h(t)dt
\end{equation*}
where $1\leq m\leq N$, $i_{1}<i_{2}<\ldots<i_{m}$ and
$l_{1},l_{2},\ldots,l_{m}\in\{1,2\}$. A simple arithmetic argument
based on equation $(\ref{zal})$ shows that the integral equals $0$
unless $l_{1}=l_{2}=\ldots=l_{m}=2$. Indeed, it is equal to
$\widehat{h}(0)$ and to prove this assertion let us assume on the
contrary $l_{s}=1$ for some $1\leq s\leq m$. If we have
$\widehat{h}(0)\neq 0$, then there exist integers
$j_{1},j_{1}',j_{2},j_{2}',\ldots,j_{m},j_{m}'$ satisfying
$j_{d},j_{d}'\in\{-2^{n_{i_{d}}}+1,\ldots,2^{n_{i_{d}}}-1\}\setminus\{0\}$
for $d=1,2,\ldots,m$ with exception of $d=s$ for which $j_{s}$
belongs to the same set of integers but $j_{s}'=0$, such that
\begin{equation*}
0=\sum_{k=1}^{m}(m_{k}(j_{k}+j_{k}')+r_{k}(\mathrm{sgn}(j_{k})+\mathrm{sgn}(j_{k}'))).
\end{equation*}
Equation $(\ref{zal})$ implies that this is possible if, and only
if, $j_{k}+j_{k}'=0$ for all $k$. However, this situation is
excluded by the assumption $l_{s}=1$, which leads to the vanishing
of number $j_{s}'$. Putting all this information together we
obtain
\begin{equation*}
||f_{N}||_{L^{2}(\mathbb{T})}^{2}=\int\prod_{k=1}^{N}(1-w_{k}(t))^{2}dt=\int\prod_{k=1}^{N}(1+w_{k}^{2}(t))dt.
\end{equation*}
Forgetting the terms with an order higher than two we have, by Proposition $\ref{postpr}$,
\begin{equation*}
\int\prod_{k=1}^{N}(1+w_{k}^{2}(t))dt\geq 1+\sum_{k=1}^{N}\int w_{k}^{2}(t)dt=1+\sum_{k=1}^{N}2\varepsilon_{k}^{2}(2^{n_{k}}-1)).
\end{equation*}
Using the assumption $(\ref{zal1})$ we finally get
\begin{equation*}
||f_{N}||^{2}_{L^{2}(\mathbb{T})}\geq 1+Nc\rightarrow\infty\text{ as }N\rightarrow\infty.
\end{equation*}
\end{proof}
The main result of this section states that under additional assumptions
on the sequences $\{r_{k}\}_{k=1}^{\infty}$, $\{m_{k}\}_{k=1}^{\infty}$ and
$\{n_{k}\}_{k=1}^{\infty}$ the resulting measure is singular. We also show how to satisfy these assumptions.
\begin{tw}\label{przy}
Let $\{\varepsilon_{k}\}_{k=1}^{\infty}$ be a decreasing sequence of
positive numbers vanishing at infinity such that
$\varepsilon_{k+1}<\frac{1}{2}\varepsilon_{k}$ for all
$k\in\mathbb{N}$. Then there exist sequences
$\{r_{k}\}_{k=1}^{\infty}$, $\{m_{k}\}_{k=1}^{\infty}$,
$\{n_{k}\}_{k=1}^{\infty}$ of positive integers satisfying the
conditions
\begin{equation}
\begin{split}\label{tzal}
&r_{k}>2\sum_{j=1}^{k-1}((2^{n_{j}}-1)m_{j}+r_{j})\text{ for }k\geq 2,\\
&m_{k}>2\sum_{j=1}^{k-1}((2^{n_{j}}-1)m_{j}+r_{j})\text{ for }k\geq 2,\\
&\frac{1}{2}<\varepsilon_{k}2^{\frac{n_{k}+3}{2}}<1.
\end{split}
\end{equation}
such that the positive measure $\mu\in M_{0}(\mathbb{T})$ with norm
$1$
\begin{equation*}
\mu=\prod_{k=1}^{\infty}(1-w_{k})
\end{equation*}
is singular and satisfies
\begin{equation*}
\widehat{\mu}(\mathbb{Z})\subset\{\pm\prod_{k=1}^{m}
\varepsilon_{k}^{l_{k}}:l_{k}\in\{0,1\},m\in\mathbb{N}\}\cup\{0\}.
\end{equation*}
\end{tw}
\begin{proof}
We show first how to choose the sequence $\{n_{k}\}_{k=1}^{\infty}$.
We define  $n_k$ as the smallest integer satisfying
\begin{equation*}
2\log_{2}\frac{1}{\varepsilon_{k}}-5<n_{k}<2\log_{2}\frac{1}{\varepsilon_{k}}-3.
\end{equation*}
To satisfy $n_{k+1}>n_{k}$ it is enough to have $2\log_{2}\frac{1}{\varepsilon_{k}}-3<2\log_{2}\frac{1}{\varepsilon_{k+1}}-5$ which is equivalent to $\varepsilon_{k+1}<\frac{1}{2}\varepsilon_k$. Now we define the sequences $\{r_{k}\}_{k=1}^{\infty}$ and $\{m_{k}\}_{k=1}^{\infty}$. We put $m_{1}=r_{1}=1$ and then we choose inductively sequences
\begin{equation*}
\begin{split}
r_{k}>2\sum_{j=1}^{k-1}((2^{n_{j}}-1)m_{j}+r_{j})\text{ for }k\geq 2,\\
m_{k}>2\sum_{j=1}^{k-1}((2^{n_{j}}-1)m_{j}+r_{j})\text{ for }k\geq 2.
\end{split}
\end{equation*}
We are now in the position to apply Proposition $\ref{ly}$ and we derive that
the measure $\mu$ exists and does not belong to
$L^{2}(\mathbb{T})$ (in fact, the  assumption
$2^{n_{k}}\varepsilon_{k}^{2}>\frac{1}{32}$ leads to this result
without referring to Proposition $\ref{ly}$). The proof of the singularity
follows essentially the argument given in [BM] and [GM] for the Riesz
products.
\\
We obviously have
\begin{equation*}
\sum_{k=1}^{\infty}2^{n_{k}}\varepsilon_{k}=\infty.
\end{equation*}
Hence, we may choose a sequence $\{c_{k}\}_{k=1}^{\infty}\in
l^{2}(\mathbb{N})$ of real numbers and an increasing sequence
$\{l_{k}\}_{k=1}^{\infty}$ of integers such that
\begin{equation}\label{sum}
\sum_{l=l_{k}+1}^{l_{k+1}}c_{l}2^{n_{l}}\varepsilon_{l}=1\text{
for all $k\in\mathbb{N}$}
\end{equation}
with the additional property
\begin{equation}\label{zb1}
\sum_{l=1}^{\infty}c_{l}^{2}2^{n_{l}}<\infty.
\end{equation}
We introduce sets $A_{l}\subset\mathbb{N}$ for $l\in\mathbb{N}$
\begin{equation*}
A_{l}=\{r_{l},r_{l}+m_{l},r_{l}+2m_{l},\ldots,r_{l}+(2^{n_{l}}-1)m_{l}\}.
\end{equation*}
Clearly, we have $A_{l}\cap A_{k}=\emptyset$ for $l\neq k$ and
$|A_{l}|=2^{n_{l}}$. Moreover, we get
$\widehat{\mu}(n)=\mathrm{sgn}\widehat{\mu}(n)\varepsilon_{l}$ for
$n\in A_{l}$. Let us consider polynomials $f_{k}$ for
$k\in\mathbb{N}$ defined by the formula
\begin{equation}\label{wiel}
f_{k}(t)=\sum_{l=l_{k}+1}^{l_{k+1}}c_{l}\sum_{n\in
A_{l}}\mathrm{sgn}\widehat{\mu}(n)e^{int}.
\end{equation}
By  $(\ref{zb1})$ we have
\begin{equation}\label{zer}
||f_{k}||_{L^{2}(\mathbb{T})}^{2}=\sum_{l=l_{k}+1}^{l_{k+1}}2^{n_{l}}c_{l}^{2}\rightarrow
0\text{ as }k\rightarrow\infty.
\end{equation}
We now perform the crucial calculation of
$||f_{k}||^{2}_{L^{2}(\mathbb{T},\mu)}$.
\begin{equation*}
\begin{split}
&||f_{k}||^{2}_{L^{2}(\mathbb{T},\mu)}=\int
f_{k}(t)\overline{f_{k}(t)}d\mu=\\
&\sum_{l}2^{n_{l}}c_{l}^{2}+\sum_{l}c_{l}^{2}\sum_{\substack{n,m\in
A_{l}\\n\neq
m}}\mathrm{sgn}\widehat{\mu}(n)\mathrm{sgn}\widehat{\mu}(m)\widehat{\mu}(m-n)+\sum_{\substack{l,r\\l\neq
r}}c_{l}2^{n_{l}}\varepsilon_{l}c_{r}2^{n_{r}}\varepsilon_{r}.
\end{split}
\end{equation*}
However, $\widehat{\mu}(m-n)=0$ for $m,n\in A_{l}$, $n\neq m$
(this is true by $(\ref{tzal})$ and the construction of $\mu$ in
Proposition \ref{kons}) and the second sum vanishes. Simple
manipulations of the remaining terms give
\begin{equation*}
\begin{split}
\sum_{l}2^{n_{l}}c_{l}^{2}+&\sum_{\substack{l,r\\l\neq
r}}c_{l}2^{n_{l}}\varepsilon_{l}c_{r}2^{n_{r}}\varepsilon_{r}=\\
&\sum_{l}2^{n_{l}}c_{l}^{2}+\left(\sum_{l}2^{n_{l}}\varepsilon_{l}c_{l}\right)^{2}-\sum_{l}c_{l}^{2}2^{2{n_{l}}}\varepsilon_{l}^{2}.
\end{split}
\end{equation*}
By $(\ref{sum})$, the second term equals  $1$ and
\begin{equation*}
||f_{k}||^{2}_{L^{2}(\mathbb{T},\mu)}=1+\sum_{l=l_{k}+1}^{l_{k+1}}c_{l}^{2}2^{n_{l}}(1-2^{n_{l}}\varepsilon_{l}^{2}).
\end{equation*}
The assumption $2^{n_{l}}\varepsilon_{l}^{2}>\frac{1}{32}$ leads to
$1-2^{n_{l}}\varepsilon_{l}^{2}<\frac{31}{32}$  which, with aid of
$(\ref{zb1})$, gives
\begin{equation*}
\sum_{l=l_{k}+1}^{l_{k+1}}c_{l}^{2}2^{n_{l}}(1-2^{n_{l}}\varepsilon_{l}^{2})\rightarrow
0\text{ as }k\rightarrow\infty.
\end{equation*}
Hence
\begin{equation}\label{jed}
||f_{k}||_{L^{2}(\mathbb{T},\mu)}\rightarrow 1\text{ as
}k\rightarrow\infty.
\end{equation}
We shall now show that
\begin{equation}\label{l1}
\lim_{k\rightarrow\infty}\|f_{k}-1\|_{L^{1}(\mathbb{T},\mu).}=0
\end{equation}
Applying the Schwarz inequality we get
\begin{equation*}
\begin{split}
\left(\int|f_{k}-1|d\mu\right)^{2}\leq\int|f_{k}-1|^{2}d\mu=\int|f_{k}|^{2}d\mu-2Re\int
f_{k}d\mu+1=\\
=\int
|f_{k}|^{2}d\mu-2Re\sum_{l=l_{k}+1}^{l_{k+1}}c_{l}2^{n_{l}}\varepsilon_{l}+1\rightarrow
0\text{ as }k\rightarrow\infty.
\end{split}
\end{equation*}
The last assertion follows from $(\ref{sum})$ and $(\ref{jed})$. This
proves $(\ref{l1})$. Finally, by ($\ref{zer}$) and ($\ref{l1}$), we can find
a subsequence $(f_{k_j})$ and a Borel set
$F\subset \mathbb{T}$
of full Lebesgue measure and $\mu(F)=1$ such that

\begin{equation*}
\begin{split}
&\lim_{j\rightarrow\infty}f_{k_{j}}(t)=0\text{  for $t\in F$},\\
&\lim_{j\rightarrow\infty}f_{k_{j}}=1\text{ $\mu$ -a.e. on $ F$ }.
\end{split}
\end{equation*}
Both the normalized Lebesgue measure and
$\mu$ are positive and have norm $1$. Hence we obtain that $\mu$
is singular with respect to the Lebesgue measure which finishes the
proof.
\end{proof}

\section{Final remarks}

\begin{enumerate}
\item The components of the open set $U$ constructed in Theorem 1 tend to 0 very fast. A closer look at the proof shows that the reciprocal of the distance of $n$-th component to 0  grows as fast as the Ackermann function $A(4,2n)$, which is very fast. It is not easy to see in which places of the proof this growth could be optimalized.
  \item  Our proof gives that the continuous part of a measure satisfying the assumptions of Theorem 1 has an absolutely continuous convolution square. We do not know, however, whether this could be improved. In particular it would be interesting to show an example of an open set $U_1$ which yields the continuous measure to be absolutely continuous if it has all its Fourier coefficients in $U_1$.
\item It would also be interesting to construct an open set $U_k$
with the property that any function with Fourier coefficients from
$U_k$ has only $k$th convolution power absolutely continuous, and
such that there exists a measure with Fourier coefficients in
$U_k$ with all smaller convolution powers singular.
  \item The above property uses the fact that the sum of measures with Fourier coefficients
tending to 0 whose  convolution powers are absolutely continuous belongs to the Zafran class $\mathscr C$. We also conjecture the converse property holds true, i.e.  any measure from $\mathscr C$ could be decomposed onto an (infinite) sum of measures whose convolution powers are absolutely continuous.
  \item The crucial element in our proof was the use of the Littlewood conjecture to estimate the number of
   repetitions of any specific value taken by the Fourier coefficients. In the case of the Cantor group, the Littlewood conjecture does not hold. This fact encourages us to ask whether any
infinite Wiener - Pitt set exist for the convolution measure
algebra on the Cantor group.
\item Our Lemma 8 is a stronger
version of the second part of Theorem 15 which is taken from [GM].
Our proof uses the exact version of the Littlewood conjecture,
which was not available when [GM] was written. But our proof
differs in more aspects - it uses Bożejko - Pełczyński's invariant
local approximation property which seems to be a simpler method.
\item While Lemma 8 does not hold for torsion abelian groups,
because the Littlewood conjecture is false there, it seems likely
that Lemma 9 may be extended to this case but it would require
completely different proof.
\end{enumerate}

The authors express their sincere gratitude to the reviewer whose
thorough remarks and comments contributed significantly to the
current shape of the paper. This concerns especially the example
which was constructed under the influence of his suggestion.
\\
The research of Przemysław Ohrysko has been supported by the
Polish Ministry of Science grant no. N N201 397737 (years
2009-2012).
\\
The research of Michał Wojciechowski has been supported by NCN
grant no. N N201 607840.

\end{document}